\documentclass[11pt]{article}

\usepackage[T1]{fontenc}
\usepackage[a4paper,margin=1in]{geometry}
\usepackage{microtype}
\usepackage{amsmath,amssymb,amsthm}
\usepackage{mathtools}
\usepackage{array}
\usepackage{aliascnt}
\usepackage[hidelinks]{hyperref}
\usepackage[capitalise,noabbrev]{cleveref}
\usepackage{authblk}

\hypersetup{
  pdftitle={Moment obstructions and continuum-to-discrete bounds for checkerboard no-three-in-line sets},
  pdfauthor={Jujhar Aujla, Thomas Prellberg, Nirvair Sandhu},
  pdfsubject={Checkerboard no-three-in-line sets, line-sum moments, and fractional packing bounds},
  pdfkeywords={no-three-in-line problem, checkerboard grid, discrete tomography, second moments, linear programming}
}
\allowdisplaybreaks

\newcommand{\eps}{\varepsilon}
\newcommand{\ZZ}{\mathbb{Z}}
\newcommand{\G}{G}
\newcommand{\C}{C}
\newcommand{\Dmono}{D_{\mathrm{mono}}}
\newcommand{\Mfour}{M_4}
\newcommand{\Lmono}{L_{\mathrm{mono}}}

\theoremstyle{plain}
\newtheorem{theorem}{Theorem}[section]

\newaliascnt{proposition}{theorem}
\newtheorem{proposition}[proposition]{Proposition}
\aliascntresetthe{proposition}
\crefname{proposition}{proposition}{propositions}
\Crefname{proposition}{Proposition}{Propositions}

\newaliascnt{lemma}{theorem}
\newtheorem{lemma}[lemma]{Lemma}
\aliascntresetthe{lemma}
\crefname{lemma}{lemma}{lemmas}
\Crefname{lemma}{Lemma}{Lemmas}

\newaliascnt{corollary}{theorem}
\newtheorem{corollary}[corollary]{Corollary}
\aliascntresetthe{corollary}
\crefname{corollary}{corollary}{corollaries}
\Crefname{corollary}{Corollary}{Corollaries}

\theoremstyle{definition}
\newaliascnt{definition}{theorem}
\newtheorem{definition}[definition]{Definition}
\aliascntresetthe{definition}
\crefname{definition}{definition}{definitions}
\Crefname{definition}{Definition}{Definitions}

\newaliascnt{example}{theorem}
\newtheorem{example}[example]{Example}
\aliascntresetthe{example}
\crefname{example}{example}{examples}
\Crefname{example}{Example}{Examples}

\theoremstyle{remark}
\newaliascnt{remark}{theorem}
\newtheorem{remark}[remark]{Remark}
\aliascntresetthe{remark}
\crefname{remark}{remark}{remarks}
\Crefname{remark}{Remark}{Remarks}

\title{Moment Obstructions and Continuum-to-Discrete Bounds\\
for Checkerboard No-Three-in-Line Sets}

\author[1]{Jujhar Aujla}
\author[2]{Thomas Prellberg}
\author[3]{Nirvair Sandhu}

\affil[1]{University of Toronto\\
\href{mailto:jujhar.aujla@mail.utoronto.ca}{\texttt{jujhar.aujla@mail.utoronto.ca}}}
\affil[2]{School of Mathematical Sciences, Queen Mary University of London\\
\href{mailto:t.prellberg@qmul.ac.uk}{\texttt{t.prellberg@qmul.ac.uk}}}
\affil[3]{University of the Fraser Valley\\
\href{mailto:nirvair.sandhu@student.ufv.ca}{\texttt{nirvair.sandhu@student.ufv.ca}}}
\date{}

\begin{document}
\maketitle

\begin{abstract}
Fix one colour class in the checkerboard colouring of an $n\times n$ integer
grid, and let $\Mfour(n,\eps)$ be the largest subset having at most two points
in every row, column, and diagonal of slopes $\pm1$.  We prove the
near-saturation bound $\Mfour(n,\eps)\leq2n-4$ for $n\geq6$.  The proof uses
first and second moments of the four line families: a hypothetical set of
size $2n-3$ produces row, column, and diagonal deficits whose exact moment
identities contradict Cauchy--Schwarz.  A finite argument handles $n=6$.

The same identity extends to arbitrary deficit multisets.  It gives
$\Mfour(n,\eps)\leq2n-d$ whenever $d\geq4$ is an integer and
$n\geq3d-4$, and an entirely discrete asymptotic estimate
\[
  \Mfour(n,\eps)\leq(\sqrt{21}-3)n+8.
\]
We also prove a general continuum-to-discrete theorem for the associated
four-direction fractional packing problem.  Applying it to the exact
continuum dual certificate constructed in earlier work yields, for both
colours,
\[
  \Lmono(n,\eps)\leq\alpha n+O(1),
  \qquad \alpha\approx1.5768233968738,
\]
and hence the same upper bound for $\Mfour$ and for checkerboard
no-three-in-line sets.
\end{abstract}



\section{Introduction}

The classical no-three-in-line problem asks for the largest number of points
that can be selected from an $n\times n$ integer grid with no three on one
Euclidean line.  It originated in Dudeney's recreational problems
\cite[Problem~317]{dudeney-1917}; early systematic treatments include
Adena, Holton and Kelly~\cite{adena-holton-kelly-1974}, Guy and
Kelly~\cite{guy-kelly-1968}, and Hall, Jackson, Sudbery and
Wild~\cite{hall-et-al-1975}.  A long computational line includes work of
Anderson~\cite{anderson-1979}, Flammenkamp~\cite{flammenkamp-1992,flammenkamp-1998},
and, more recently, comparisons of integer programming and learning-based
heuristics~\cite{ramanathan-et-al-2025} and a constraint-programming
formulation~\cite{prellberg-csp-2026}. General background appears in
Brass, Moser and Pach~\cite[Section~10.1]{brass-moser-pach-2005}; the related problem of
selecting points in general position from an arbitrary planar point set is
studied in~\cite{froese-et-al-2017}. The elementary row bound is $2n$, while
the best known asymptotic construction has $(3/2-o(1))n$
points~\cite{hall-et-al-1975}; the asymptotic behaviour remains unresolved.
For the related unrestricted problem allowing at most $k$ points on each
line, Ghosal et al.~\cite{ghosal-et-al-2026} prove that the maximum is $kn$
for every fixed $k\geq3$ and sufficiently large $n$. Their result leaves
the classical case $k=2$ open.

We study the checkerboard-restricted variant introduced in
\cite{prellberg-checkerboard-2026}.  Put $N=n-1$ and
\[
  \C_\eps=\{(x,y)\in\{0,\ldots,N\}^2:x+y\equiv\eps\pmod2\},
  \qquad \eps\in\{0,1\}.
\]
Write $\Dmono(n,\eps)$ for the largest no-three-in-line subset of
$\C_\eps$, and let $\Dmono(n)$ denote the larger of the two colour-class
optima.  Call a subset of $\C_\eps$ \emph{four-direction admissible} if
every row, column, difference diagonal, and sum diagonal contains at most two
of its points, and let $\Mfour(n,\eps)$ be the corresponding maximum.  Every
no-three-in-line set is four-direction admissible, so
\[
  \Dmono(n,\eps)\leq\Mfour(n,\eps).
\]
Because the slope-$\pm1$ diagonals are monochromatic, their total capacities
give the elementary upper bound $2n-2$ for $n\geq2$.

\paragraph{Relationship with the earlier checkerboard paper.}
The present paper is a companion to~\cite{prellberg-checkerboard-2026}.
That paper introduced the checkerboard problem, the four-direction linear-programming relaxation and its symmetry
reductions, and an exact continuum dual certificate with objective
$\alpha\approx1.5768233968738$.  Its first arXiv version suggested the bound
$2n-4$ by a boundary forcing argument but explicitly left it unproved, and it
did not transfer the continuum certificate to the finite grids.  The new
contributions here are: a proof of $2n-4$ for the stronger four-direction
problem; a multiple-deficit extension yielding both fixed-deficit and
linear-deficit bounds; and a continuum-to-discrete sampling theorem that
applies to both checkerboard colours.  The revised version
of~\cite{prellberg-checkerboard-2026} records the resulting bound
$\Dmono(n)\leq\alpha n+O(1)$, citing \cref{thm:transference} below and
sketching its proof in the odd-fat case, and cites
\cref{thm:main-intro} for the bound $2n-4$ in place of the earlier heuristic.
The moment arguments and the general sampling theorem
are independent of the particular continuum certificate.  Only the
application giving the coefficient $\alpha$ uses the existence of the
feasible profiles established in~\cite[Theorem~1]{prellberg-checkerboard-2026}.
Their construction is due to that earlier work; a separate exact verification
is provided in the ancillary files. The ancillary files also give independently
checkable certificates for the finite values $\Dmono(n,\eps)$, for both
colours and $2\leq n\leq16$, reported there.

The moment identities used below also have a natural place in the literature
on discrete tomography.  For the four directions $0,\infty,+1,-1$, the
relevant total, first-moment, and second-moment identities are global
line-sum dependencies; see Hajdu and Tijdeman~\cite[Remark~4]{hajdu-tijdeman-2001}
and Stolk and Batenburg~\cite[Section~2.2]{stolk-batenburg-2010}.
The novelty here lies in
applying these consistency relations to checkerboard capacity deficits and
combining them with extremal inequalities.

Our first result excludes the elementary capacity bound $2n-2$ and the
next smaller cardinality $2n-3$.

\begin{theorem}[Near-saturation obstruction]\label{thm:main-intro}
For every $n\geq6$ and $\eps\in\{0,1\}$,
\[
  \Mfour(n,\eps)\leq2n-4.
\]
Consequently $\Dmono(n,\eps)\leq2n-4$.
\end{theorem}

The threshold is best possible: for $n=5$, one colour contains an eight-point
no-three-in-line set, attaining $2n-2$.  Allowing an arbitrary number of
missing capacity units gives two further consequences.

\begin{theorem}[Multiple-deficit bounds]\label{thm:moment-extensions-intro}
For either checkerboard colour $\eps\in\{0,1\}$:
\begin{enumerate}
\item if $d\geq4$ is an integer and $n\geq3d-4$, then
\[
  \Mfour(n,\eps)\leq2n-d;
\]
\item for every $n\geq4$,
\[
  \Mfour(n,\eps)\leq(\sqrt{21}-3)n+8.
\]
\end{enumerate}
The same bounds hold for $\Dmono(n,\eps)$.
\end{theorem}

The coefficient $\sqrt{21}-3\approx1.5825756949$ comes entirely from the
discrete second-moment argument.  A slightly sharper coefficient follows
from the fractional problem.  Let $\Lmono(n,\eps)$ denote the associated
four-direction fractional packing optimum, defined in
\cref{sec:transference}.  Since characteristic vectors of four-direction
admissible sets are feasible,
\[
  \Dmono(n,\eps)\leq\Mfour(n,\eps)\leq\Lmono(n,\eps).
\]

\begin{theorem}[Fractional asymptotic bound]\label{thm:asymptotic-intro}
Let $\alpha$ be the middle real root of
\[
  401\alpha^3-1744\alpha^2+2240\alpha-768=0.
\]
There is an absolute constant $C$ such that, for every $n\geq2$ and
$\eps\in\{0,1\}$,
\[
  \Dmono(n,\eps)\leq\Mfour(n,\eps)
  \leq\Lmono(n,\eps)\leq\alpha n+C.
\]
In particular,
\[
  \limsup_{n\to\infty}\frac{\Mfour(n,\eps)}n
  \leq\alpha,
  \qquad
  1.576823396873<\alpha<1.576823396874.
\]
\end{theorem}

The result is one-sided: it proves neither optimality of the continuum
certificate nor a matching lower bound.

\paragraph{Subsequent correspondence and related work.}
The circulation of~\cite{prellberg-checkerboard-2026} prompted several
readers to investigate the proposed $2n-4$ bound and brought related
computational work to our attention. The present collaboration began on
21 July 2026, when Aujla and Sandhu contacted Prellberg with a first- and
second-moment argument for the four-direction $2n-4$ bound. In his reply that day, Prellberg proposed extending the method
to multiple deficits, including the fixed-deficit bounds and the asymptotic
coefficient $\sqrt{21}-3$. These exchanges led to the joint work developed
here.

Subsequent communications described related approaches to the same finite
bound. Zurnaci~\cite{zurnaci-communication-2026} supplied a centred moment
argument. Chen~\cite{chen-communication-2026} supplied explicit quadratic
line weights, reporting the asymptotic coefficient $2^{2/3}$, together with
arguments for the finite bound. Payzin~\cite{payzin-communication-2026}
supplied a rational piecewise-linear continuum cover and a discretisation
argument reporting the bound $(324n+2244)/205$, supplemented by finite
certificates and enumeration for the smaller grids. These communications
illustrate the usefulness of moment identities and line covers for this
problem. For comparison, these coefficients and those of the present paper
satisfy
\[
  \alpha<\frac{324}{205}<\sqrt{21}-3<2^{2/3}.
\]
These descriptions concern unpublished work communicated to us.

Pochuev~\cite{pochuev-solver-2026} and
Oliveira~\cite{oliveira-communication-2026} also communicated computational
work reporting the values $26,27,29,30$ for $\Dmono(n)$ at
$n=17,18,19,20$, respectively, using constructions, rational line-cover
certificates and exhaustive searches. Pochuev's public repository records
these computations and reports exact results through $n=21$, with $n=22$
still open as of 10 September 2026. A separate public release by
Cohen~\cite{cohen-n23-2026} reports $\Dmono(23)=35$.
Motivated by Pochuev's results,
Prellberg also undertook exploratory stochastic SAT searches and communicated
configurations for larger grids in his reply of 16 July 2026, without
claiming optimality. These finite investigations complement the uniform and
asymptotic bounds studied here.

\paragraph{Organisation.}
The paper is organised as follows.
\Cref{sec:setting} fixes the notation and elementary capacities.
\Cref{sec:moment-proof} proves the near-saturation theorem.
\Cref{sec:multiple-deficits} develops the centred deficit identity and its
two consequences.  \Cref{sec:transference} defines the fractional relaxation,
proves the general continuum-to-discrete theorem and applies it to the
earlier continuum bound.  \Cref{sec:limitations} discusses the remaining
questions.

\section{Setting and elementary structure}\label{sec:setting}

Fix $n\geq1$, put $N=n-1$, and write
\[
  \G_n=\{0,1,\ldots,N\}^2.
\]
For $\eps\in\{0,1\}$, the corresponding checkerboard class is
\[
  \C_\eps=\{(x,y)\in\G_n:x+y\equiv\eps\pmod2\}.
\]

\begin{definition}
A set $S\subseteq\ZZ^2$ is \emph{no-three-in-line} if no three distinct
points of $S$ are collinear. Define
\[
  \Dmono(n,\eps)=
  \max\{|S|:S\subseteq\C_\eps\text{ is no-three-in-line}\},
\]
and set $\Dmono(n)=\max_{\eps\in\{0,1\}}\Dmono(n,\eps)$.
We call $S\subseteq\C_\eps$ \emph{four-direction admissible} if every row,
column, difference diagonal $x-y=d$, and sum diagonal $x+y=s$ contains at
most two points of $S$, and define
\[
  \Mfour(n,\eps)=
  \max\{|S|:S\subseteq\C_\eps\text{ is four-direction admissible}\}.
\]
\end{definition}

Every no-three-in-line set is four-direction admissible, although the
converse need not hold. Hence
\[
  \Dmono(n,\eps)\leq\Mfour(n,\eps).
\]
The moment argument will work entirely within this elementary relaxation.

Only parameters $d\equiv\eps\pmod2$ and $s\equiv\eps\pmod2$ meet
$\C_\eps$. Their lengths in the grid are
\[
  N+1-|d|\quad(-N\leq d\leq N),
  \qquad
  N+1-|s-N|\quad(0\leq s\leq2N).
\]
To include the singleton boundary diagonals in the same notation, define
\[
  c_d(d)=\min\{2,N+1-|d|\},
  \qquad
  c_s(s)=\min\{2,N+1-|s-N|\}.
\]

\begin{lemma}[Diagonal capacity]\label{lem:capacity}
For $n\geq2$ and each $\eps\in\{0,1\}$,
\[
  \sum_{\substack{-N\leq d\leq N\\d\equiv\eps\ (2)}}c_d(d)
  =
  \sum_{\substack{0\leq s\leq2N\\s\equiv\eps\ (2)}}c_s(s)
  =2N.
\]
Consequently every four-direction admissible set has at most $2N=2n-2$
points.
\end{lemma}

\begin{proof}
Consider first the difference diagonals. If $\eps\equiv N\pmod2$, there are
$N+1$ compatible parameters, and the two extreme diagonals $d=\pm N$ have
capacity one. Their total capacity is therefore $2(N+1)-2=2N$. Otherwise
there are $N$ compatible parameters, each of capacity two. For the sum
diagonals, colour zero contains the singleton extremes $s=0,2N$, whereas
colour one does not; the same count again gives total capacity $2N$.
Summing the occupancy bounds over either diagonal family proves the final
assertion.
\end{proof}

The condition $n\geq2$ is necessary. For $n=1$, the point $(0,0)$ gives
$\Dmono(1,0)=1$, whereas $2n-2=0$.

\begin{example}[Why the threshold cannot begin at five]
For $n=5$, the colour-one set
\[
\begin{split}
\{&(0,1),(0,3),(1,0),(1,4),\\
  &(3,0),(3,4),(4,1),(4,3)\}
\end{split}
\]
is no-three-in-line and has $8=2n-2$ points. Thus the conclusion of
\cref{thm:main-intro} fails at $n=5$.
\end{example}

The symmetries of the square will also be used in the fractional relaxation
below. The identity, the half-turn, and the reflections in $x=y$ and
$x+y=N$ preserve $x+y$ modulo two. Each of the remaining four symmetries
adds $N$ modulo two. Thus, when $n$ is odd, all eight symmetries preserve
each colour, whereas when $n$ is even, four of them interchange the two
colours. In particular, $\Dmono(n,0)=\Dmono(n,1)$ for even $n$.

\section{The four-direction moment obstruction}\label{sec:moment-proof}

The argument proceeds in three stages. Near saturation first leaves a small
set of integral deficits in each of the four line families. Their first two
moments then produce an exact identity. Finally, Cauchy--Schwarz and concavity give a lower bound incompatible
with that identity. This comparison handles $n\geq7$; the remaining case
$n=6$ is settled by the integer structure of the deficit triples.

\begin{remark}[Discrete-tomography viewpoint]\label{rem:tomography}
For a finitely supported weight function on $\ZZ^2$, its line sums in the
four directions $0,\infty,+1,-1$ satisfy global consistency relations of
degrees zero, one, and two.  In the present coordinates, the degree-one
relations compare the first moments of row and column sums with those of the
two diagonal families, while the degree-two relation is the line-sum form of
\[
  2x^2+2y^2=(x-y)^2+(x+y)^2.
\]
These dependencies are standard in discrete tomography; for the four
directions used here, see Hajdu and Tijdeman~\cite[Remark~4]{hajdu-tijdeman-2001}
and Stolk and Batenburg~\cite[Section~2.2]{stolk-batenburg-2010}. A later
explicit construction treats these four directions in
\cite[Example~4.1]{hajdu-tijdeman-2017}. Our use of them is
extremal rather than reconstructive: we apply the relations to nearly
saturated checkerboard capacity profiles and combine the resulting deficit
identities with convexity and order estimates.
\end{remark}

\subsection{Reduction and deficit triples}

Assume, for a contradiction, that a four-direction admissible set
$S\subseteq\C_\eps$ has at least $2N-1$ points. By \cref{lem:capacity}, its
size is either $2N-1$ or $2N$. In the latter case, delete one point. It
therefore suffices to rule out
\begin{equation}\label{eq:size-near-saturated}
  |S|=2N-1.
\end{equation}

Each compatible diagonal family has total capacity $2N$, so its nonnegative
integer deficits sum to one. Hence there is a unique difference diagonal
$d_0$ and a unique sum diagonal $s_0$ whose occupancy falls one short of its
capacity; every other compatible diagonal is saturated. This includes the
possibility that the deficient diagonal is a singleton and is therefore
empty.

Every row and column contains at most two points of $S$. Relative to
two nominal slots in each of the $N+1$ columns,
\eqref{eq:size-near-saturated} leaves exactly three vacant slots. Record their
column indices, with multiplicity, as
\[
  \mathbf a=(a_1,a_2,a_3).
\]
Thus a column of occupancy one contributes its index once, and an empty
column contributes it twice. Define the row-deficit triple
$\mathbf b=(b_1,b_2,b_3)$ in the same way. No index can occur three times in
either triple. Conversely, any triple with multiplicities at most two
determines an occupancy vector, although it need not arise from a
geometrically realisable point set. Allowing all such triples is therefore a
valid relaxation.

Put
\[
  X=\sum_{i=1}^3a_i,
  \quad Q_a=\sum_{i=1}^3a_i^2,
  \qquad
  Y=\sum_{i=1}^3b_i,
  \quad Q_b=\sum_{i=1}^3b_i^2,
\]
and
\[
  T_2=\sum_{j=0}^N j^2=\frac{N(N+1)(2N+1)}6.
\]
Subtracting the vacant slots from two nominal slots on every coordinate line
gives
\begin{align}
  \sum_{(x,y)\in S}x&=N(N+1)-X,
  &\sum_{(x,y)\in S}x^2&=2T_2-Q_a,\label{eq:coordinate-moments-x}\\
  \sum_{(x,y)\in S}y&=N(N+1)-Y,
  &\sum_{(x,y)\in S}y^2&=2T_2-Q_b.\label{eq:coordinate-moments-y}
\end{align}

\subsection{First and second diagonal moments}

\begin{lemma}[First moments]
The deficit sums satisfy
\begin{equation}\label{eq:first-moment-identities}
  X-Y=d_0,
  \qquad
  X+Y=2N+s_0.
\end{equation}
In particular,
\begin{equation}\label{eq:first-moment-region}
  |X-Y|\leq N,
  \qquad
  2N\leq X+Y\leq4N.
\end{equation}
\end{lemma}

\begin{proof}
The capacity-weighted first moment of the difference parameters vanishes by
the involution $d\mapsto-d$. Removing the missing capacity unit at $d_0$
therefore gives
\[
  \sum_{S}(x-y)=-d_0.
\]
By \cref{eq:coordinate-moments-x,eq:coordinate-moments-y}, the left-hand side
is $Y-X$, and hence $X-Y=d_0$.

For the sum diagonals, the capacities are symmetric about $N$ and have total
$2N$, so their capacity-weighted first moment is $2N^2$. Removing the unit of
capacity missing at $s_0$ gives
\[
  \sum_{S}(x+y)=2N^2-s_0.
\]
The coordinate identities make the left-hand side
$2N(N+1)-X-Y$, proving the second equality. The parameter ranges
$-N\leq d_0\leq N$ and $0\leq s_0\leq2N$ then yield
\eqref{eq:first-moment-region}.
\end{proof}

Define the capacity-weighted diagonal square moment
\[
  K_{N,\eps}=
  \sum_{\substack{-N\leq d\leq N\\d\equiv\eps\ (2)}}c_d(d)d^2
  +
  \sum_{\substack{0\leq s\leq2N\\s\equiv\eps\ (2)}}c_s(s)s^2.
\]

\begin{lemma}[Diagonal square sums]
For $N\geq1$,
\begin{equation}\label{eq:K-formulas}
K_{N,\eps}=
\begin{cases}
\dfrac{2N(5N^2+1)}3,
  &N\text{ odd},\\[5pt]
\dfrac{2N(5N^2+4)}3,
  &N\text{ even and }\eps=0,\\[5pt]
\dfrac{2N(5N^2-2)}3,
  &N\text{ even and }\eps=1.
\end{cases}
\end{equation}
\end{lemma}

\begin{proof}
Start with weight two on every compatible parameter, and then correct for the
singleton extremes. This gives
\begin{equation}
\frac{K_{N,\eps}}2=
\sum_{d\equiv\eps\ (2)}d^2+
\sum_{s\equiv\eps\ (2)}s^2
-N^2\mathbf1_{\eps\equiv N\ (2)}
-2N^2\mathbf1_{\eps=0},
\end{equation}
where the sums are taken over the parameter ranges in the definition of
$K_{N,\eps}$. Let $Q(m)=m(m+1)(2m+1)/6$. Before simplification, the four
cases of $K_{N,\eps}/2$ are
\[
\begin{array}{c|c|l}
N&\eps&K_{N,\eps}/2\\
\hline
2m+1&0&8Q(m)+4Q(2m+1)-2N^2\\
2m+1&1&2\sum_{j=0}^{m}(2j+1)^2+
          \sum_{j=0}^{2m}(2j+1)^2-N^2\\
2m&0&8Q(m)+4Q(2m)-3N^2\\
2m&1&2\sum_{j=0}^{m-1}(2j+1)^2+
        \sum_{j=0}^{2m-1}(2j+1)^2.
\end{array}
\]
Using
$\sum_{j=0}^{r}(2j+1)^2=(r+1)(4r^2+8r+3)/3$
and simplifying gives \eqref{eq:K-formulas}.
\end{proof}

For a triple $\mathbf z=(z_1,z_2,z_3)$, define
\begin{equation}
  F_N(\mathbf z)=\sum_{i=1}^3z_i^2-
  \left(\sum_{i=1}^3z_i-N\right)^2.
\end{equation}

\begin{lemma}[Exact moment invariant]\label{lem:moment-invariant}
The deficit triples satisfy
\begin{equation}
  F_N(\mathbf a)+F_N(\mathbf b)=R_{N,\eps},
\end{equation}
where
\begin{equation}\label{eq:R-formulas}
R_{N,\eps}=
\begin{cases}
\dfrac{N(-N^2+6N+1)}3,
  &N\text{ odd},\\[5pt]
\dfrac{N(-N^2+6N-2)}3,
  &N\text{ even and }\eps=0,\\[5pt]
\dfrac{N(-N^2+6N+4)}3,
  &N\text{ even and }\eps=1.
\end{cases}
\end{equation}
\end{lemma}

\begin{proof}
The only absent capacity units occur at $d_0$ and $s_0$. Since
$(x-y)^2+(x+y)^2=2x^2+2y^2$, the second coordinate moments give
\begin{equation}\label{eq:raw-second-moment}
  K_{N,\eps}-d_0^2-s_0^2=8T_2-2(Q_a+Q_b).
\end{equation}
By \eqref{eq:first-moment-identities},
\[
  \frac{d_0^2+s_0^2}{2}=(X-N)^2+(Y-N)^2.
\]
Substituting this into \eqref{eq:raw-second-moment} and rearranging yields
\[
  F_N(\mathbf a)+F_N(\mathbf b)=4T_2-\frac{K_{N,\eps}}2.
\]
The formulas in \eqref{eq:K-formulas} now give \eqref{eq:R-formulas}.
\end{proof}

\subsection{The universal lower bound}

\begin{lemma}\label{lem:universal-lower-bound}
Any triples satisfying \eqref{eq:first-moment-region} obey
\[
  F_N(\mathbf a)+F_N(\mathbf b)\geq\frac{N^2}{3}.
\]
\end{lemma}

\begin{proof}
By Cauchy--Schwarz, $Q_a\geq X^2/3$, and therefore
\[
  F_N(\mathbf a)\geq\frac{X^2}{3}-(X-N)^2
  =N^2 f(X/N),
\]
where
\[
  f(t)=\frac{t^2}{3}-(t-1)^2=-\frac23t^2+2t-1.
\]
The same estimate holds for $\mathbf b$. Set $u=X/N$ and $v=Y/N$. By
\eqref{eq:first-moment-region}, $(u,v)$ belongs to the parallelogram
\[
  P=\{(u,v):2\leq u+v\leq4,\ |u-v|\leq1\}.
\]
Its vertices are
\[
  (1/2,3/2),\ (3/2,1/2),\ (3/2,5/2),\ (5/2,3/2).
\]
The function $(u,v)\mapsto f(u)+f(v)$ is concave. Since every point of $P$
is a convex combination of its vertices, concavity bounds its value below by
the minimum of its values at those vertices. Each vertex gives $1/3$.
Consequently,
\[
  F_N(\mathbf a)+F_N(\mathbf b)
  \geq N^2(f(u)+f(v))\geq\frac{N^2}{3}.
\]
\end{proof}

\subsection{Contradiction for \texorpdfstring{$n\geq7$}{n at least 7}
and the integer case \texorpdfstring{$n=6$}{n = 6}}

\begin{proof}[Proof of the main finite theorem]
Suppose first that $n$ is odd. Then $N$ is even and $N\geq6$. The larger of
the two values in \eqref{eq:R-formulas} is the one for $\eps=1$, and
\[
  \frac{N(-N^2+6N+4)}3<\frac{N^2}{3}
\]
is equivalent to $N^2-5N-4>0$, which holds for every even $N\geq6$.
If $n$ is even and $n\geq8$, then $N$ is odd and $N\geq7$; in this case the
required strict inequality is equivalent to $N^2-5N-1>0$. Thus, for every
$n\geq7$ and either colour, \cref{lem:moment-invariant} contradicts
\cref{lem:universal-lower-bound}.

It remains to treat $n=6$, so $N=5$. Here both colours satisfy
$R_{5,\eps}=10$. A deficit index cannot occur three times. For each possible
sum $Z$, let $\mu(Z)$ be the minimum of $F_5$ over triples in
$\{0,\ldots,5\}$ with multiplicity at most two.  The table lists
$\mu(Z)$ and one sorted minimiser; for example, $001$ denotes $(0,0,1)$.
\begin{equation}\label{eq:n6-mu-table}
\begin{array}{c|rrrrrrrrrrrrrr}
Z&1&2&3&4&5&6&7&8&9&10&11&12&13&14\\
\hline
\mu(Z)&-15&-7&1&5&9&13&13&13&13&9&5&1&-7&-15\\
\text{triple}&001&011&012&112&122&123&223&233&234&334&344&345&445&455
\end{array}
\end{equation}
For fixed $Z$, the sum of squares is minimised by making the three entries as
balanced as possible. If the balanced triple is constant, and hence
forbidden, the next minimum is obtained by replacing $(k,k,k)$ with
$(k-1,k,k+1)$. This proves the table.

The first-moment constraints become
\[
  10\leq X+Y\leq20,
  \qquad |X-Y|\leq5.
\]
If necessary, complement both triples by $z\mapsto5-z$, and then interchange
them. Complementation preserves each $F_5$ value and sends $X+Y$ to
$30-(X+Y)$. We may therefore assume $X\leq Y$ and
$10\leq X+Y\leq15$. The only possibilities are
\[
\begin{array}{c|ccccc}
X&3&4&5&6&7\\
\hline
Y&7,8&6,7,8,9&5,\ldots,10&6,7,8,9&7,8.
\end{array}
\]
By \eqref{eq:n6-mu-table}, the smallest possible value of
$\mu(X)+\mu(Y)$ is $14$, contradicting the required value
$R_{5,\eps}=10$.

We have therefore ruled out a set of size $2N-1$ for every $N\geq5$. Any set
of size $2N$ would produce one of size $2N-1$ after deletion of a point.
Hence every four-direction admissible set has size at most
$2N-2=2n-4$, proving \cref{thm:main-intro}.
\end{proof}

\begin{corollary}
For both colours,
\[
  \Dmono(6,0)=\Dmono(6,1)=8.
\]
\end{corollary}

\begin{proof}
The upper bound follows from \cref{thm:main-intro}. Exact lower witnesses are
\[
\begin{split}
S_0=\{&(0,2),(0,4),(1,1),(1,5),\\
      &(4,0),(4,2),(5,3),(5,5)\},
\end{split}
\]
and
\[
\begin{split}
S_1=\{&(0,1),(0,5),(2,1),(3,0),\\
      &(4,3),(4,5),(5,0),(5,2)\}.
\end{split}
\]
Every point of $S_\eps$ has colour $\eps$. For each set, direct evaluation of
the integer determinant
\[
  (x_2-x_1)(y_3-y_1)-(y_2-y_1)(x_3-x_1)
\]
on all $\binom83=56$ triples yields no zero.
\end{proof}

\section{Multiple deficits and linear bounds}\label{sec:multiple-deficits}

The preceding proof uses the most rigid near-saturated situation: each
diagonal capacity profile is short by one unit. The underlying identity does
not depend on that restriction. When several units are missing, centring the
coordinates puts the same calculation into a cleaner form involving four
deficit multisets. We draw two consequences. The first forces any prescribed
fixed deficit once $n$ is sufficiently large. The second records where the
deficits occur and, using only second moments, forces a deficit linear in
$n$.

\subsection{A centred deficit identity}

Assume $n\geq4$ and again put $N=n-1$. For a four-direction admissible set
$S\subseteq\C_\eps$, write
\begin{equation}\label{eq:multi-size}
  |S|=2N-r.
\end{equation}
By \cref{lem:capacity}, $0\leq r\leq2N$.

For each column, place its index in a multiset $A$ once for every unused unit
of its nominal capacity two. Define the row-deficit multiset $B$ in the same
way. There are $2(N+1)$ nominal slots in the rows and the same number in the
columns, so
\begin{equation}\label{eq:multi-AB-size}
  |A|=|B|=r+2.
\end{equation}

For the diagonal families, use the centred labels
\[
  u=x-y,
  \qquad
  v=x+y-N.
\]
Both range from $-N$ to $N$. Let $D$ and $E$ denote the multisets of unused
capacity units in the $u$- and $v$-families, respectively, with each centred
label repeated according to the amount of missing capacity there. Since each
diagonal family has total capacity $2N$,
\begin{equation}\label{eq:multi-DE-size}
  |D|=|E|=r.
\end{equation}

Write
\begin{equation}\label{eq:multi-T}
  T_N=\sum_{j=0}^{N}\left(j-\frac N2\right)^2
      =\frac{N(N+1)(N+2)}{12}.
\end{equation}
A centred diagonal family has one of only two possible capacity-weighted
square moments. If its compatible labels include $\pm N$, then
\begin{equation}\label{eq:multi-Mplus}
  M_+(N)=2\sum_{j=0}^{N}(2j-N)^2-2N^2
        =\frac{2N(N^2+2)}3.
\end{equation}
If the endpoints are incompatible, then
\begin{equation}\label{eq:multi-Mminus}
  M_-(N)=2\sum_{j=0}^{N-1}(2j-(N-1))^2
        =\frac{2N(N^2-1)}3.
\end{equation}
The $u$-family uses $M_+(N)$ exactly when $\eps\equiv N\pmod2$, while the
$v$-family uses $M_+(N)$ exactly when $\eps=0$.

\begin{proposition}[Centred deficit identity]\label{prop:multi-identity}
For every four-direction admissible $S$ satisfying
\eqref{eq:multi-size},
\begin{align}
 &\frac12\left(\sum_{d\in D}d^2+\sum_{e\in E}e^2\right)
 -\sum_{a\in A}\left(a-\frac N2\right)^2
 -\sum_{b\in B}\left(b-\frac N2\right)^2
 \notag\\*
 &\hspace{42mm}=\Gamma_{N,\eps},              \label{eq:multi-identity}
\end{align}
where
\begin{equation}\label{eq:multi-Gamma}
\Gamma_{N,\eps}=
\begin{cases}
\displaystyle \frac{N(N-1)(N-2)}3,
 &N\text{ even and }\eps=0,\\[6pt]
\displaystyle \frac{N(N-4)(N+1)}3,
 &N\text{ even and }\eps=1,\\[6pt]
\displaystyle \frac{N(N^2-3N-1)}3,
 &N\text{ odd}.
\end{cases}
\end{equation}
\end{proposition}

\begin{proof}
For every point $(x,y)$,
\begin{equation}\label{eq:multi-point-identity}
  (x-y)^2+(x+y-N)^2
  =2\left(x-\frac N2\right)^2
   +2\left(y-\frac N2\right)^2.
\end{equation}
Let $M_u$ and $M_v$ be the full capacity-weighted square moments of the two
centred diagonal families. Removing the diagonal deficits gives
\[
  \sum_{(x,y)\in S}(u^2+v^2)
  =M_u+M_v-\sum_{d\in D}d^2-\sum_{e\in E}e^2.
\]
On the coordinate side,
\[
  \sum_{(x,y)\in S}\left(x-\frac N2\right)^2
  =2T_N-\sum_{a\in A}\left(a-\frac N2\right)^2,
\]
and the same formula holds with $x,A$ replaced by $y,B$. Summing
\eqref{eq:multi-point-identity} over $S$ and rearranging yields
\eqref{eq:multi-identity}, where
\[
  \Gamma_{N,\eps}=\frac{M_u+M_v}{2}-4T_N.
\]
Substituting \eqref{eq:multi-T}-\eqref{eq:multi-Mminus}, according to which
families contain the endpoint labels, gives \eqref{eq:multi-Gamma}.
\end{proof}

When $r=1$, this is precisely the second-moment obstruction of
\cref{sec:moment-proof}, expressed in centred coordinates. For general $r$,
we combine the identity with inequalities that use the restricted locations
of the deficit units.

\subsection{Fixed deficits}

The first consequence requires almost no information about those locations.
The coordinate terms in \eqref{eq:multi-identity} are nonnegative, while
every centred diagonal label has absolute value at most $N$. Therefore
\begin{equation}\label{eq:multi-crude}
  \Gamma_{N,\eps}\leq rN^2.
\end{equation}
This crude estimate already forces every fixed deficit.

\begin{theorem}[Fixed deficits]\label{thm:multi-fixed}
Let $d\geq4$ be an integer. If $n\geq3d-4$, then
\[
  \Mfour(n,\eps)\leq2n-d
\]
for either colour $\eps$. Consequently,
\[
  \Dmono(n,\eps)\leq2n-d.
\]
\end{theorem}

\begin{proof}
Suppose, to the contrary, that $|S|\geq2n-d+1$. Deleting points preserves
four-direction admissibility, so we may assume equality. Since $n=N+1$,
\eqref{eq:multi-size} gives
\[
  r=d-3.
\]
Uniformly over the parity of $N$ and the colour $\eps$,
\eqref{eq:multi-Gamma} yields
\begin{equation}\label{eq:multi-Gamma-lower}
  \Gamma_{N,\eps}\geq\frac{N(N-4)(N+1)}3
  =\frac{N^3}{3}-N^2-\frac{4N}{3}.
\end{equation}
The hypothesis $n\geq3d-4$ is equivalent to $N\geq3d-5$. Hence
\begin{align*}
  \Gamma_{N,\eps}-(d-3)N^2
  &\geq\frac N3\bigl(N(N-3d+6)-4\bigr)\\
  &\geq\frac{N(N-4)}3>0.
\end{align*}
This contradicts \eqref{eq:multi-crude}.
\end{proof}

\begin{corollary}\label{cor:multi-third}
For every $n\geq8$ and either colour,
\[
  \Mfour(n,\eps)\leq
  2n-\left\lfloor\frac{n+4}{3}\right\rfloor,
\]
and the same bound holds for $\Dmono(n,\eps)$.
\end{corollary}

\begin{proof}
Take $d=\lfloor(n+4)/3\rfloor$. Then $d\geq4$ and $3d-4\leq n$, so
\cref{thm:multi-fixed} applies.
\end{proof}

The threshold $3d-4$ is chosen for a uniform statement. Keeping the parity
and colour cases in \eqref{eq:multi-Gamma} separate gives slightly sharper
finite thresholds, but the simpler form is more useful for the present
purpose.

\paragraph{Sharpness.}
The bound in \cref{thm:main-intro} is attained for $n=6,7,8$ in both
colours and for $n=9$ in colour zero. The fixed-deficit bound in
\cref{thm:multi-fixed} is also attained at $(n,d)=(11,5)$ in colour zero
and at $(n,d)=(14,6)$ in both colours:
\[
  \Mfour(11,0)=17,\qquad
  \Mfour(14,0)=\Mfour(14,1)=22.
\]
The ancillary files give coordinate witnesses and exact checks of their
four-direction admissibility. The order-11 and order-14 witnesses contain
collinear triples in other directions, so these equalities concern
$\Mfour$, rather than $\Dmono$.

\subsection{A linear deficit from the deficit positions}

The estimate \eqref{eq:multi-crude} loses information in two places: it
places every missing diagonal unit as far from the centre as possible, and it
discards the contribution of the row and column deficits altogether. Both
losses can be controlled by ordering the available positions.

Let $P_N(k)$ be the sum of the $k$ smallest elements of the multiset
\begin{equation}\label{eq:multi-P-multiset}
  \left\{
  \left(j-\frac N2\right)^2,
  \left(j-\frac N2\right)^2:0\leq j\leq N
  \right\}.
\end{equation}
For $\sigma\in\{u,v\}$, let $Q_\sigma(N,r)$ be the sum of the $r$ largest
squared labels in the full capacity multiset of the corresponding diagonal
family. The centred identity then gives
\begin{equation}\label{eq:multi-rearrangement}
  \Gamma_{N,\eps}
  \leq\frac{Q_u(N,r)+Q_v(N,r)}2-2P_N(r+2).
\end{equation}

\begin{lemma}[Order-statistic bounds]\label{lem:multi-orders}
For every integer $r$ with $0\leq r\leq2N$ and either diagonal family,
\begin{align}
  Q_\sigma(N,r)
  &\leq I_N(r)+4N^2,
  & I_N(r)&=N^2r-\frac{Nr^2}{2}+\frac{r^3}{12},
                                                        \label{eq:multi-Q-bound}\\
  P_N(r+2)
  &\geq\frac{r(r+1)(2r+1)}{96}
   \geq\frac{r^3}{48}.                                \label{eq:multi-P-bound}
\end{align}
The estimates are uniform in $N,r,\eps$ and in the choice of diagonal family.
\end{lemma}

\begin{proof}
Order the absolute values in a diagonal capacity multiset as
\[
  w_1\geq w_2\geq\cdots\geq w_{2N}.
\]
For $0\leq t\leq N$, there are at most $2(N-t)+4$ capacity units with
absolute label at least $t$.  Taking $t=w_j$ shows that
$j\leq2(N-w_j)+4$, and hence
\[
  w_j\leq N-\frac{j-4}{2}\qquad(5\leq j\leq2N).
\]
Therefore the four largest terms contribute at most $4N^2$.  Since
$x\mapsto(N-x/2)^2$ is decreasing on $[0,2N]$, each term of the remaining
sum is bounded by the integral over the preceding unit interval.  Thus
\begin{align*}
  Q_\sigma(N,r)
  &\leq4N^2+
  \sum_{j=1}^{\max\{r-4,0\}}\left(N-\frac j2\right)^2\\
  &\leq4N^2+\int_0^r\left(N-\frac x2\right)^2\,dx
   =4N^2+I_N(r).
\end{align*}

Next order the absolute centred coordinate offsets, with multiplicity two,
as
\[
  0\leq z_1\leq z_2\leq\cdots\leq z_{2N+2}.
\]
An interval of radius $t$ about $N/2$ contains at most $2t+1$ integer
indices, hence at most $4t+2$ capacity units.  It follows that
\[
  z_{j+2}\geq\frac j4\qquad(1\leq j\leq2N).
\]
Therefore
\[
  P_N(r+2)\geq\frac1{16}\sum_{j=1}^{r}j^2
  =\frac{r(r+1)(2r+1)}{96},
\]
which proves \eqref{eq:multi-P-bound}.
\end{proof}

\begin{theorem}[A linear deficit]\label{thm:multi-linear}
For every $n\geq4$, either colour $\eps$, and every four-direction admissible
$S\subseteq\C_\eps$,
\begin{equation}\label{eq:multi-linear-finite}
  |S|\leq(\sqrt{21}-3)n+8.
\end{equation}
Consequently,
\[
  \limsup_{n\to\infty}\frac{\Mfour(n,\eps)}n
  \leq\sqrt{21}-3=1.5825756949\ldots,
\]
and the same statement holds for $\Dmono(n,\eps)$.
\end{theorem}

\begin{proof}
Write $|S|=2N-r$ and set $\rho=r/N$. Combining
\eqref{eq:multi-Gamma-lower}, \eqref{eq:multi-rearrangement}, and
\cref{lem:multi-orders} gives
\[
  \frac{N^3}{3}-N^2-\frac{4N}{3}
  \leq I_N(r)-\frac{r^3}{24}+4N^2.
\]
Since
\[
  I_N(r)-\frac{r^3}{24}=N^3f(\rho),
  \qquad
  f(\rho)=\rho-\frac{\rho^2}{2}+\frac{\rho^3}{24},
\]
we obtain, for $N\geq3$,
\begin{equation}\label{eq:multi-f-bound}
  f(\rho)\geq\frac13-\frac5N-\frac{4}{3N^2}
  \geq\frac13-\frac{49}{9N}.
\end{equation}
The smallest nonnegative solution of $f(\rho)=1/3$ is
\[
  \rho_0=5-\sqrt{21},
\]
because
\[
  24\left(f(\rho)-\frac13\right)
  =(\rho-2)(\rho^2-10\rho+4).
\]
Moreover, on $[0,\rho_0]$,
\[
  f'(\rho)\geq f'(\rho_0)=\frac{7-\sqrt{21}}4.
\]
If $\rho\geq\rho_0$, the required lower bound on $r$ already holds. If
$\rho<\rho_0$, the mean value theorem together with
\eqref{eq:multi-f-bound} gives
\[
  \rho_0-\rho\leq\frac{7(7+\sqrt{21})}{9N}.
\]
Hence
\[
  r\geq(5-\sqrt{21})N-\frac{7(7+\sqrt{21})}{9}.
\]
Using $|S|=2N-r$ and $N=n-1$, we conclude that
\begin{align*}
  |S|
  &\leq(\sqrt{21}-3)N+\frac{7(7+\sqrt{21})}{9}\\
  &=(\sqrt{21}-3)n+\frac{76-2\sqrt{21}}9\\
  &<(\sqrt{21}-3)n+8.
\end{align*}
\end{proof}

The additive constant in \cref{thm:multi-linear} is not optimised. It arises
from the deliberately coarse $4N^2$ allowance in
\eqref{eq:multi-Q-bound}. Replacing that estimate by exact order statistics
would improve the constant without changing the coefficient
$\sqrt{21}-3$.

The coefficient $\sqrt{21}-3$ is slightly larger than the continuum
coefficient $\alpha$. The comparison separates the roles of the two
arguments: the discrete second-moment method already forces a linear deficit,
while the dual certificate improves the coefficient. The two methods are
therefore complementary rather than redundant.

\section{Fractional relaxation and continuum-to-discrete transfer}
\label{sec:transference}

\subsection{The four-direction fractional packing problem}

Assign a nonnegative mass $z_p$ to each point $p\in\C_\eps$, impose total
mass at most two on every row, column, difference diagonal, and sum diagonal,
and maximise $\sum_p z_p$.  Denote the optimum by $\Lmono(n,\eps)$.  We do
not impose separate pointwise constraints $z_p\leq1$; omitting them enlarges
the feasible region and therefore preserves the validity of the resulting
upper bound.  The characteristic vector of every four-direction admissible
set is feasible, so
\begin{equation}\label{eq:relaxation-chain}
  \Dmono(n,\eps)\leq\Mfour(n,\eps)\leq\Lmono(n,\eps).
\end{equation}

The dual assigns nonnegative weights $h_y,v_x,\lambda_d,\mu_s$ to the four
line families, where $0\leq x,y\leq n-1$,
$-(n-1)\leq d\leq n-1$, and $0\leq s\leq2n-2$, and minimises
\begin{equation}\label{eq:generic-dual-objective}
  2\left(
    \sum_{y=0}^{n-1}h_y+\sum_{x=0}^{n-1}v_x
    +\sum_{d=-(n-1)}^{n-1}\lambda_d
    +\sum_{s=0}^{2n-2}\mu_s
  \right)
\end{equation}
subject to
\begin{equation}\label{eq:generic-dual-cover}
  h_y+v_x+\lambda_{x-y}+\mu_{x+y}\geq1
  \qquad((x,y)\in\C_\eps).
\end{equation}
Finite linear-programming duality identifies this minimum with
$\Lmono(n,\eps)$; weak duality alone is sufficient for all upper bounds used
below~\cite{schrijver-1986}.

\subsection{Odd symmetry reductions}

For odd $n=2m+1$, the full square symmetry preserves each colour class.  We
call $\C_0$ the fat class and $\C_1$ the thin class.  The next reduced duals
were derived in~\cite{prellberg-checkerboard-2026}; we include the short
orbit argument because the precise half-mesh shift is needed in the sampling
theorem.

\begin{proposition}[Odd reduced duals]\label{prop:odd-duals}
For every integer $m\geq1$, the fat-class value $\Lmono(2m+1,0)$ is the
minimum of
\begin{equation}\label{eq:fat-objective}
  \Phi_m^{\mathrm{fat}}(a,b)=
  8\sum_{i=0}^{m-1}(a_i+b_i)+4(a_m+b_m)
\end{equation}
over nonnegative $a_0,\ldots,a_m,b_0,\ldots,b_m$ satisfying
\begin{equation}\label{eq:fat-constraint}
  a_u+a_{m-v}+b_{u+v}+b_{u-v}\geq1
  \quad(0\leq v\leq u,\ u+v\leq m).
\end{equation}
For the thin class, $\Lmono(2m+1,1)$ is the minimum of
\begin{equation}\label{eq:thin-objective}
  \Phi_m^{\mathrm{thin}}(a,b)=
  8\sum_{i=0}^{m-1}a_i+8\sum_{i=0}^{m-1}b_i+4b_m
\end{equation}
over nonnegative $a_0,\ldots,a_{m-1},b_0,\ldots,b_m$ satisfying
\begin{equation}\label{eq:thin-constraint}
  a_u+a_{m-v-1}+b_{u+v+1}+b_{u-v}\geq1
  \quad(0\leq v\leq u,\ u+v\leq m-1).
\end{equation}
\end{proposition}

\begin{proof}
Use the standard averaging argument for symmetric linear
programmes~\cite[Proposition~13]{bodi-herr-joswig-2013}:
average a feasible generic dual over the colour-preserving dihedral
symmetries.  This preserves feasibility and objective value and makes the
weights constant on line orbits. For either colour $\eps$, define the orbit
weights by
\[
  h_t=v_t=b_{\min\{t,2m-t\}},\qquad
  \lambda_d=a_{m-(|d|+\eps)/2},\qquad
  \mu_s=a_{m-(|s-2m|+\eps)/2},
\]
where the diagonal formulas apply to compatible labels
$d,s\equiv\eps\pmod2$. Incompatible diagonals contain no point of the
selected colour and may be assigned weight zero.

Reflect a point into $[0,m]^2$ and interchange its coordinates if necessary
to obtain $0\leq Y\leq X\leq m$. In the fat class, putting
$u=(X+Y)/2$ and $v=(X-Y)/2$ gives representatives
$(X,Y)=(u+v,u-v)$ with $0\leq v\leq u$ and $u+v\leq m$. The four
incident orbit weights are $a_u,a_{m-v},b_{u+v},b_{u-v}$, giving
\eqref{eq:fat-constraint}.  Noncentral line orbits have four members and
central orbits two; the factor two in \eqref{eq:generic-dual-objective} gives
the coefficients eight and four in \eqref{eq:fat-objective}.

For the thin class, putting $u=(X+Y-1)/2$ and $v=(X-Y-1)/2$ gives
representatives $(X,Y)=(u+v+1,u-v)$ with $0\leq v\leq u$ and
$u+v\leq m-1$. The half-layer shift changes the four indices to those in
\eqref{eq:thin-constraint}.  There is no central $a$-orbit, while $b_m$
remains central, giving \eqref{eq:thin-objective}.  Conversely, assigning a
reduced variable to each line in its orbit recovers a feasible generic dual,
so the reduced and generic optima agree.
\end{proof}

\subsection{Sampling a continuum dual cover}

Let
\[
  T=\{(x,y):0\leq y\leq x,\ x+y\leq1\}.
\]
Suppose that nonnegative functions $A,B:[0,1]\to[0,\infty)$ satisfy the
continuum obstacle
\begin{equation}\label{eq:continuum-obstacle}
  A(x)+A(1-y)+B(x+y)+B(x-y)\geq1
  \qquad((x,y)\in T).
\end{equation}
The inequality is required on the closed triangle, including its boundary,
since boundary points occur in the sampling below.

The following elementary quadrature estimate is the equally spaced
one-dimensional case of Koksma's inequality; see, for example,
\cite[equation~(2)]{aistleitner-dick-2015}. We include a direct proof.

\begin{lemma}[Riemann sums for bounded-variation functions]
\label{lem:bv-riemann}
If $f$ has bounded variation on $[0,1]$, then, for every integer $m\geq1$,
\[
  \left|\sum_{i=0}^{m-1}f(i/m)-m\int_0^1f(t)\,dt\right|
  \leq\operatorname{Var}(f)
\]
and
\[
  \left|\sum_{i=0}^{m-1}f((i+1/2)/m)-m\int_0^1f(t)\,dt\right|
  \leq\operatorname{Var}(f),
\]
where $\operatorname{Var}(f)$ denotes the total variation on $[0,1]$.
\end{lemma}

\begin{proof}
On an interval $I_i=[i/m,(i+1)/m]$, the difference between any sampled value
of $f$ and the average value of $f$ on $I_i$ is at most the variation of $f$
on $I_i$.  Multiplying by the interval length and summing shows that the
normalised Riemann sum differs from the integral by at most
$\operatorname{Var}(f)/m$.  Multiplication by $m$ gives both statements.
\end{proof}

\begin{theorem}[Continuum-to-discrete transference]\label{thm:transference}
Suppose nonnegative bounded-variation functions $A,B$ satisfy
\eqref{eq:continuum-obstacle}, and put
\[
  \beta=4\int_0^1(A(t)+B(t))\,dt.
\]
Then, for every $n\geq2$ and each $\eps\in\{0,1\}$,
\[
  \Lmono(n,\eps)\leq\beta n+C_{A,B},
  \qquad
  C_{A,B}=8\bigl(\operatorname{Var}A+\operatorname{Var}B\bigr)
           +4\bigl(A(1)+B(1)\bigr).
\]
Consequently,
\[
  \Mfour(n,\eps)\leq\beta n+O(1),
  \qquad
  \Dmono(n,\eps)\leq\beta n+O(1),
\]
and the corresponding normalised upper limits are at most $\beta$.
The implied constants may depend on the fixed profiles $A,B$, but are
independent of $n$ and may be chosen uniformly in $\eps$.
\end{theorem}

\begin{proof}
For the fat class, set
\[
  a_i=A(i/m),\qquad b_i=B(i/m)\qquad(0\leq i\leq m).
\]
For each constraint \eqref{eq:fat-constraint}, take $x=u/m$ and $y=v/m$.
The index conditions place $(x,y)$ in $T$, and the constraint left side is
exactly
\[
  A(x)+A(1-y)+B(x+y)+B(x-y),
\]
so the sampled variables are feasible.  By \cref{lem:bv-riemann}, their
objective is
\[
\begin{split}
  8\sum_{i=0}^{m-1}(A+B)(i/m)+4(A(1)+B(1))
  &=8m\int_0^1(A+B)+O(1)\\
  &=2m\beta+O(1)\\
  &=\beta(2m+1)+O(1).
\end{split}
\]

For the thin class, use the shifted samples
\[
  a_i=A((i+1/2)/m)\quad(0\leq i<m),
  \qquad
  b_i=B(i/m)\quad(0\leq i\leq m).
\]
For a constraint \eqref{eq:thin-constraint}, set
\[
  x=\frac{u+1/2}{m},
  \qquad
  y=\frac{v+1/2}{m}.
\]
Then $0\leq y\leq x$ and $x+y\leq1$, while the left side becomes exactly
\[
  A(x)+A(1-y)+B(x+y)+B(x-y).
\]
Thus the samples are feasible, and the midpoint and left Riemann sums give
\[
\begin{split}
  &8\sum_{i=0}^{m-1}A((i+1/2)/m)
  +8\sum_{i=0}^{m-1}B(i/m)+4B(1)\\
  &\hspace{40mm}=2m\beta+O(1)
  =\beta(2m+1)+O(1).
\end{split}
\]
This proves the two odd-grid estimates.

Finally, extend any feasible primal mass assignment on $\G_n$ by zero on the
new row and column of $\G_{n+1}$.  Every old line sum is unchanged and every
new-only sum is zero, so
\[
  \Lmono(n,\eps)\leq\Lmono(n+1,\eps).
\]
The variation bounds in \cref{lem:bv-riemann} show that
both odd-grid objectives are at most $\beta(n-1)+C_{A,B}$. Since
$\beta\geq0$, this is at most $\beta n+C_{A,B}$; for even $n$, comparison
with $n+1$ gives the same bound. Thus the stated constant works for every
$n\geq2$ and both colours.
The bounds for $\Mfour$ and $\Dmono$ now follow from
\eqref{eq:relaxation-chain}.
\end{proof}

\subsection{Application of the earlier continuum bound}\label{sec:alpha-application}

The general result above applies to any profiles satisfying its hypotheses.
For the numerical coefficient $\alpha$, we use the following input from
\cite[Theorem~1 and Section~4.2]{prellberg-checkerboard-2026}: there exist
continuous nonnegative piecewise-polynomial functions $A,B$ satisfying
\eqref{eq:continuum-obstacle} and
\[
  4\int_0^1(A+B)=\alpha,
\]
where $\alpha$ is the middle real root of
$401\alpha^3-1744\alpha^2+2240\alpha-768$.  Since these functions are
piecewise polynomial, they have bounded variation. These properties are
the only features of the earlier construction needed here. A separate
exact verification of the profiles is provided in the ancillary files.

\begin{corollary}\label{cor:alpha}
There is an absolute constant $C$ such that, for every $n\geq2$ and
$\eps\in\{0,1\}$,
\[
  \Dmono(n,\eps)\leq\Mfour(n,\eps)
  \leq\Lmono(n,\eps)\leq\alpha n+C.
\]
In particular,
\[
  \limsup_{n\to\infty}\frac{\Lmono(n,\eps)}n
  \leq\alpha,
  \qquad
  \limsup_{n\to\infty}\frac{\Mfour(n,\eps)}n
  \leq\alpha.
\]
\end{corollary}

\begin{proof}
Apply \cref{thm:transference} to the profiles supplied by
\cite[Theorem~1]{prellberg-checkerboard-2026}, taking $C=C_{A,B}$.
The chain of inequalities is
\eqref{eq:relaxation-chain}.
\end{proof}

\section{Concluding remarks}\label{sec:limitations}

The three bounds expose different levels of the four-direction structure.
The $2n-4$ theorem is a moment obstruction to near saturation.  The
multiple-deficit argument remains entirely discrete and already forces a
deficit linear in $n$.  The continuum sampling theorem gives the smaller
coefficient $\alpha$, but its numerical value comes from the exact feasible
certificate of~\cite{prellberg-checkerboard-2026}; neither that certificate's
optimality nor convergence of the normalised finite LP optima is known.
There is also no matching construction of no-three-in-line sets of size
$\alpha n-o(n)$.

Several directions remain open.  Exact order statistics would improve the
additive constant in \cref{thm:multi-linear}, though not its coefficient.
More substantially, one may seek an optimal continuum primal--dual pair or
incorporate further slope families, which may provide additional moment
dependencies.
The discrete-tomography viewpoint in \cref{rem:tomography} suggests a
systematic language for this extension; explicit constructions of global
dependencies for general direction sets are given
in~\cite{hajdu-tijdeman-2017}.

\subsection*{Acknowledgements}
We are grateful for the interest in the earlier checkerboard paper and for
the care with which readers explored its open questions and shared their
ideas. We thank Baris Zurnaci, Jialin Chen and Cemal Payzin for communicating
their approaches to the proposed finite bound, and Grisha Pochuev and Luan
Oliveira for sharing their computational investigations. This exchange,
including the correspondence that led to the present collaboration, has
been an encouraging part of the development of this work.

\subsection*{Author contributions}

Sandhu developed the one-deficit moment argument proving the $2n-4$ bound and
the continuum-to-finite sampling argument.  Aujla helped check the
one-deficit proof and suggested the separate treatment of the exceptional
case $n=6$.  Prellberg developed the entire multiple-deficit argument,
including the fixed-deficit bounds and the $\sqrt{21}-3$ estimate.  The exact continuum
certificate used in \cref{cor:alpha} is from Prellberg's earlier
work~\cite{prellberg-checkerboard-2026}.  The authors jointly checked the
results and prepared the manuscript.

\subsection*{Declaration of generative AI and AI-assisted technologies}

During the development of this work, generative AI tools, including OpenAI
ChatGPT, were used as exploratory and drafting aids.  They assisted in
comparing possible arguments, stress-testing proof ideas, checking algebraic
manipulations and intermediate calculations, and organising exposition and
LaTeX.  The overall research direction and mathematical substance were
determined by the authors.  AI-generated suggestions were not treated as
mathematical evidence or authority, and every proof step, computation,
reference, and final passage was independently reviewed by an author.  The
authors take responsibility for the content of the manuscript.

\end{document}